\documentclass[11pt]{amsart}
\usepackage[utf8]{inputenc}
\usepackage{amsfonts}
\usepackage{amssymb}
\usepackage{amsmath}
\usepackage{delarray}
\usepackage{graphicx}
\newtheorem{thm}{Theorem}[section]
\newtheorem{theorem}[thm]{Theorem}

\newtheorem{lemma}[thm]{Lemma}
\newtheorem{conjecture}[thm]{Conjecture}

\newtheorem{defn}[thm]{Definition}
\theoremstyle{remark}
\newtheorem{remark}[thm]{Remark}
\numberwithin{equation}{section}
\newtheorem{example}[thm]{Example}

\newcommand{\cF}{\mathcal F}

\newcommand{\cU}{\mathcal U}

\newcommand{\bbR}{\mathbb R}
\newcommand{\bbT}{\mathbb T}
\newcommand{\bbC}{\mathbb C}

\newcommand{\bbZ}{\mathbb Z}

\newcommand{\rank}{{\rm rank\ }}

\begin{document}

\title{Nondegenerate singularities of integrable dynamical systems}

\author{Nguyen Tien Zung}
\address{Institut de Mathématiques de Toulouse, UMR5219, Université Toulouse 3}
\email{tienzung.nguyen@math.univ-toulouse.fr}

\begin{abstract}{%
We give a natural notion of nondegeneracy for singular points of integrable non-Hamiltonian systems, 
and show that such nondegenerate singularities are locally geometrically linearizable and deformation rigid
in the analytic case. We conjecture that the same result also holds in the smooth case,  and prove this conjecture 
for systems of type $(n,0)$, i.e. $n$ commuting smooth vector fields on a $n$-manifold.
}\end{abstract}

\date{This version: June 2013, accepted for publication in Ergodic Theory and Dynamical Systems}

\maketitle

\section{Introduction}

There are many natural dynamical systems which are non-Hamiltonian, maybe because they have non-holonomic constraints or 
because they don't conserve the energy, etc., but which are still integrable in a natural sense, see, e.g.
\cite{BaCu-Nonholonomic1999,  Bogoyavlenskij-Integrability1998,  FedorovKozlov-Suslin, Stolovitch1} for some examples.
It is an interesting question to study the topology, and in particular the singularities, of such integrable non-Hamiltonian systems.
Unlike the Hamiltonian case, which has been very extensively studied, the non-Hamiltonian case is still largely open.
To our knowledge, even the notion of nondegeneracy of singularities for integrable non-Hamiltonian systems has not appeared in the
literature before. 

The aim of this paper is to establish this notion of nondegeneracy, and to study it. In particular, we want to
extend geometric local linearization theorems of Vey  \cite{Vey}  and Eliasson \cite{Eliasson-Normal1990}
to the non-Hamiltonian case. We will show that, similarly
to the Hamiltonian case, nondegenerate singularities of analytic integrable dynamical systems 
are rigid with respect to deformations, and are  geometrically linearizable
(see Theorem \ref{thm:linearization-nondegenerate}, Theorem \ref{thm:linearization-nondegenerate2} and Theorem \ref{thm:rigid}). We conjecture
that the same theorem is also true for smooth non-Hamiltonian integrable systems, and prove this conjecture
for the class of systems of type $(n,0)$, i.e. $n$ commuting smooth vector fields on a $n$-manifold 
(Theorem \ref{thm:NormalForm}). This last theorem is the starting point of a very recent work by 
Nguyen Van Minh and the author \cite{ZungMinh_Rn2012} 
on the geometry of nondegenerate $\bbR^n$-actions on $n$-manifolds.

In this paper, we will work in both the analytic (real or complex) and the smooth categories. The analytic part of this paper
relies heavily on our theorem  \cite{Zung-Birkhoff2005,Zung-Poincare2002} on the existence of 
convergent Poincaré--Dulac--Birkhoff  normalization for analytic integrable dynamical systems.

\section{Geometric equivalence of integrable systems}

Let us recall that, a dynamical system given by a vector field $X$ on a $m$-dimensional 
manifold $M$ is called {\bf integrable} (in the non-Hamiltonian sense)
if there exist integers $p \geq 1, q \geq 0$, $p+q = m$,  
$p$ vector fields $X_1 = X, X_2, \hdots, X_p,$ and $q$ functions $F_1,\hdots,F_q$ on $M,$ 
such that the vector fields $X_1,\hdots,X_p$ commute with each other, and the functions $F_1,\hdots, F_q$ 
are common first integrals for these vector fields:
\begin{equation}
[X_i, X_j] = 0  \ \forall \ i,j=1,\hdots p 
\end{equation}
and
\begin{equation}
  X_i(F_j) = 0 \ \forall \ i=1,\hdots,p, \ j = 1,\hdots, q.
\end{equation}
Moreover, one requires that 
\begin{equation}
dF_1 \wedge \hdots \wedge dF_q \neq 0 \ \text{and} \ \  X_1 \wedge \hdots \wedge X_p \neq 0
\end{equation}
almost everywhere.  We will also
say that the $m$-tuple $(X_1,\hdots,X_p,F_1,\hdots,F_q)$ is an {\bf integrable system of type $(p,q)$}. 
 This notion of non-Hamiltonian integrability is a very natural extension of the notion of integrability à la Liouville from the
Hamiltonian case to the non-Hamiltonian case, and it retains the main dynamical features of Hamiltonian integrability, see, e.g. 
\cite{AyoulZung-Galois2010, BaCu-Nonholonomic1999,  Bogoyavlenskij-Integrability1998,  FedorovKozlov-Suslin, Stolovitch1, 
Zung-Poincare2002,Zung-Torus2006}. A Hamiltonian system with $n$ degrees of freedom which is integrable à la Liouville is also
integrable in the above sense with $p=q = n$ and $m=2n.$

Geometrically, an integrable system $(X_1,\hdots,X_p,F_1,\hdots,F_p)$ 
of type $(p,q)$ may be viewed as a singular $p$-dimensional foliation (given by the
infinitesimal $\mathbb{K}^p$-action generated by $X_1,\hdots, X_p$, where $\mathbb{K} = \mathbb{R}$ or $\mathbb{C}$,
and moreover each leaf of this foliation admits an natural
induced affine structure from the action), and this foliation is integrable in the sense that 
it admits a complete set of first integrals, i.e. the functional dimension of the algebra of first integrals of the foliation is equal
to the codimension of the foliation. 

Denote by $\mathcal{F}$ the algebra of common first integrals of $X_1,\hdots,X_p.$ 
Instead of taking $F_1,\hdots, F_q$, we can choose from $\mathcal{F}$ any other family of $q$ functionally independent functions, and they will
still form with $X_1,\hdots, X_p$ an integrable system. Moreover, in general, there is no natural preferred choice of $q$ functions in
$\mathcal{F}$.  So, instead of specifying $q$ first integrals, sometimes it is better to look at the whole algebra 
$\mathcal{F}$ of first integrals.

Notice also that, if $f_{ij} \in \mathcal{F}$ ($i,j=1,\hdots,p$)  such that the matrix $(f_{ij})$ is invertible, then  by putting
\begin{equation}
 \hat{X_i} = \sum_{ij} f_{ij} X_j \ \ \text{for all} \ \ i=1,\hdots,p,
\end{equation}
we get another integrable system $(\hat{X_1},\hdots,\hat{X_p},F_1,\hdots,F_q)$, which, from the geometric point of view, is essentially
the same as the original system, because it gives rise to the same integrable singular foliation, and the same affine structure  on 
the leaves of the foliation.

\begin{defn} \label{defn:GeometricEquivalence}
Two integrable dynamical  systems  $(X_1,\hdots,X_p,F_1,\hdots,F_q)$ and $(X'_1,\hdots,X'_p,F'_1,\hdots,F'_q)$
of type $(p,q)$ on a manifold $M$ are said to be {\bf geometrically equal}, if they have the same algebra of first 
integrals (i.e. $F'_1,\hdots,F'_p$ are functionally dependent of $F_1,\hdots,F_p$ and vice versa), and there exists
a  matrix $(f_{ij})_{i=1,\hdots,p}^{j=1,\hdots,p}$, whose entries $f_{ij}$ are first integrals of the system, and whose
determinant is non-zero everywhere, such that one can write 
\begin{equation}
X'_i = \sum_{j}f_{ij} X_j \ \ \forall \ i=1,\hdots,p. 
\end{equation}
Two integrable systems are said to be {\bf geometrically equivalent} if they become geometrically the same 
after a diffeomorphism. 
\end{defn}

In this paper, we will be mainly interested in the local structure of integrable dynamical systems, up to geometric equivalence, in the sense
of the above definition. It's clear that, near a regular point, i.e. a point $z$ such that $X_1 \wedge \hdots \wedge X_p (z) \neq 0$, any two
integrable systems of the same type $(p,q)$ will be locally geometrically equivalent, and is equivalent to the rectified
system $X_1 = \frac{\partial}{\partial x_1}, \hdots, X_p = \frac{\partial}{\partial x_p}$. The question about the local structure becomes
interesting only at singular points.  Remark also that, in the definition of geometric equivalence, we don't really care about the choice
of first integrals $F_1,\hdots,F_q$ and can change them by other functionally independent first integrals at will.

If $X_1 \wedge \hdots \wedge X_p (z) = 0$ but $X_{k+1} \wedge \hdots \wedge X_p (z) \neq 0$ for example, then we can simultaneously
rectify  $X_{k+1}, \hdots, X_p$, i.e. find a coordinate system in which
\begin{equation}
X_{k+1} = \frac{\partial}{\partial x_1}, \hdots, X_{p} = \frac{\partial}{\partial x_{p-k}}.  
\end{equation}
Then the system does not depend on the coordinates $x_1,\hdots, x_{p-k}$,  and we can reduce
the problem to that of a system of type $(k,q)$ by forgetting about $x_1, \hdots, x_{p-k}$ and $X_{k+1},\hdots, X_p$. After such a reduction,
we may assume that $z$ is a fixed point of the system, i.e. all the vector fields of the system vanish at $z$. 
The situation  is similar to that
of integrable Hamiltonian systems, where the local study of singular points can also be reduced to the study of fixed points.

\section{Linear integrable systems}

Let $(X_1,\hdots, X_p, F_1,\hdots, F_q)$ be an integrable system of type $(p,q)$ on a manifold $M$, and assume that $z \in M$ is a fixed point 
of the system, i.e.  $X_1(z) = \hdots = X_p(z) = 0$. Fix a local coordinate system around $z$. Denote by $Y_i$ the linear part of $X_i$ at $z$, and by 
$G_j$ the homogeneous part  (i.e. the non-constant terms of lowest degree in the Taylor expansion)  
of $F_j$, with respect to the above coordinate system. Then, the first terms of the Taylor expansion of the identities  $[X_i,X_k] = 0$ 
and $X_i(F_j) = 0$ show that the vector fields $Y_1,\hdots, Y_p$ commute with each other and have $G_1,\hdots, G_q$ as common first 
integrals. Hence, $(Y_1,\hdots,Y_p, G_1,\hdots, G_q)$ is again an integrable system of type $(p,q)$, provided that the independence conditions 
$Y_1 \wedge \hdots  \wedge Y_p \neq 0$ and $dG_1 \wedge \hdots \wedge dG_q \neq 0$ (almost everywhere) still hold. 

The above observations lead to the following definition:

\begin{defn}
 An integrable system $(Y_1,\hdots,Y_p, G_1,\hdots, G_q)$ of type $(p,q)$ is called {\bf linear} with respect to a given coordinate system
if the vector fields $Y_1,\hdots Y_p$ are linear and the functions $G_1,\hdots, G_q$ are homogeneous. If, moreover, it is obtained from another
integrable system $(X_1,\hdots, X_p, F_1,\hdots, F_q)$ by the above construction, then we will say that $(Y_1,\hdots,Y_p, G_1,\hdots, G_q)$
 is the {\bf linear part} of  the system $(X_1,\hdots, X_p, F_1,\hdots, F_q)$. If all the vector fields $Y_1,\hdots,Y_p$
are semisimple, then we will say that $(Y_1,\hdots,Y_p, G_1,\hdots, G_q)$ is a
{\bf nondegenerate linear integrable system}.
\end{defn}

Recall that the set of linear vector fields on $\mathbb{K}^m$, 
where $\mathbb{K}= \mathbb{R}$ or $\mathbb{C}$, is a Lie algebra which is 
naturally isomorphic to $gl(m,\mathbb{K})$. 
Any linear vector field admits a unique decomposition into the sum of its semisimple part and nilpotent
part (the Jordan decomposition), and it can be diagonalized over $\mathbb{C}$ if any only if it's semisimple, i.e. its nilpotent part is zero. It is also
well-known that if we have a family of commuting semisimple elements of $gl(m,\mathbb{C})$, 
then they can be simultaneously diagonalized 
over $\mathbb{C}$. Thus, if $(Y_1,\hdots,Y_p, G_1,\hdots, G_q)$ is a nondegenerate linear integrable system, then there exists a complex 
coordinate system in which the vector fields $Y_1,\hdots Y_p$ are diagonal.

The above notion of nondegeneracy is absolutely similar to the Hamiltonian case, where one also asks that the (Hamiltonian) vector fields $Y_i$
be semisimple. It is well-known that, already in the Hamiltonian case, not every integrable
linear system is nondegenerate. 

\begin{example}
In $\mathbb{R}^4$, take $\displaystyle G_1 = x_1y_1 - x_2 y_2, G_2 = y_1y_2, 
Y_1 = x_1 \frac{\partial}{\partial x_1} - y_1 \frac{\partial}{\partial y_1} - x_2 \frac{\partial}{\partial x_2} + y_2 \frac{\partial}{\partial y_2},
Y_2 = y_2 \frac{\partial}{\partial x_1} + y_1 \frac{\partial}{\partial x_2}.$ Then this is a degenerate (non-semisimple) integrable
linear Hamiltonian system.
 \end{example}

Let $(Y_1,\hdots,Y_p, G_1,\hdots, G_q)$ be a nondegenerate linear integrable system. We will work over $\mathbb{C}$, and assume that
the coordinate system is already chosen so that the vector fields $Y_1,\hdots, Y_p$ are linear:
\begin{equation}
Y_i = \sum_{i=j}^m c_{ij} x_j\frac{\partial}{\partial x_j} .
\end{equation}
The independence condition $Y_1 \wedge \hdots \wedge Y_p \neq 0$ means that the matrix $(c_{ij})^{i=1,\hdots,p}_{j=1,\hdots,m}$ is of rank
$p$. The set of polynomial common first integrals of $Y_1,\hdots, Y_p$ is the vector space spanned by the monomial functions
$\prod_{j=1}^m x_j^{\alpha_j}$ such that
\begin{equation} \label{eqn:resonance}
\sum_{j=1}^m \alpha_j c_{ij} = 0 \
\text{for all} \  i=1,\hdots, p.
\end{equation}
This linear equation is called the {\bf resonance equation} of the vector fields $Y_1,\hdots, Y_p$.

The set of nonnegative integer solutions of the resonance equation (\ref{eqn:resonance}) is the intersection 
\begin{equation}
S \cap \mathbb{Z}^m_+, 
\end{equation}
where 
\begin{equation}
S = \left\{ (\alpha_i) \in \mathbb{R}^m \ | \  \sum_{j=1}^m \alpha_j c_{ij} = 0 \ \text{for all} \  i=1,\hdots, p \right\} 
\end{equation}
 is the $q$-dimensional
space of all real solutions of (\ref{eqn:resonance}), and $\mathbb{Z}^m_+$ is the set of nonnegative $m$-tuples of integers.
The functional independence of $G_1,\hdots,G_q$ implies that this set $S \cap \mathbb{Z}^m_+$ must have dimension 
$q$ over $\mathbb{Z}$. In particular, the set $S \cap \mathbb{R}^m_+$ has dimension $q$ over $\mathbb{R},$
and the resonance equation (\ref{eqn:resonance}) is equivalent to a linear system of equations with integer coefficients.
In other words, using a linear transformation to replace $Y_i$ by new vector fields 
\begin{equation}
\tilde{Y_i} = \sum_{j} a_{ij} Y_j 
\end{equation}
with an appropriate invertible matrix $(a_{ij})$ with constant coefficients, we may assume that 
\begin{equation}
\tilde{Y}_i = \sum_{i=j}^m \tilde{c}_{ij} x_j \frac{\partial}{ \partial x_j},
\end{equation}
where 
\begin{equation}
\tilde{c}_{ij} = \sum_{k} a_{ik} c_{kj} \in \bbZ \ \forall \ i,j .
\end{equation}
Of course, if $(Y_1,\hdots,Y_p, G_1,\hdots, G_q)$ is an integrable
system, and $\tilde{Y_i} = \sum_{j} a_{ij} Y_j$ is an invertible linear transformation of the vector fields $Y_i$, then
$(\tilde{Y_1},\hdots,\tilde{Y_p}, G_1,\hdots, G_q)$ is again in integrable system which, from the geometric point of view, is
the same as the system $(Y_1,\hdots,Y_p, G_1,\hdots, G_q)$.

Conversely, if the first integrals are not yet given, but the coefficients $c_{ij}$ are integers, 
and the set of nonnegative solutions to the resonance equation (\ref{eqn:resonance}) has
dimension $q$, then we can choose $q$ linearly independent nonnegative integer solutions of (\ref{eqn:resonance}), 
and the $q$  corresponding monomial functions will be functionally independent common first integrals of $Y_1,\hdots,Y_p$, 
and we get an integrable system.

Notice that, given a set of linear vector fields as above, the choice of common first integrals in order to turn it into an integrable system 
is far from unique. Moreover, the algebra of polynomial first integrals does not admit a set of $q$ generators in general, even though
its functional dimension is equal to $q$. The following simple example illustrates the situation:
Consider a linear integrable 4-dimensional system of type $(1,3)$, i.e. with 1 vector field and 3 functions. The vector field is
$Y = x_1 \frac{\partial}{\partial x_1} + x_2 \frac{\partial}{\partial x_2} - x_3 \frac{\partial}{\partial x_3} - x_4 \frac{\partial}{\partial x_4}.$
The corresponding resonance equation is: $\alpha_1 + \alpha_2 - \alpha_3 - \alpha_4 = 0$.  The algebra of algebraic first integrals
is generated by the functions $x_1x_3, x_1x_4, x_2,x_3, x_2x_4$; it has functional dimension 3 but cannot be generated by just 3 functions.

\begin{remark}
If $z$ is an isolated singular point of $X_1$ in an integrable system $(X_1,\hdots, X_p, F_1,\hdots, F_q)$, 
then it will be automatically a fixed point of the system. Indeed, if  $X_i(z) \neq 0$ for some $i$, then due to the commutativity of $X_1$ with 
$X_i$, $X_1$ will vanish not only at $z$, but on the whole local  trajectory of $X_i$ which goes through $z$, and so $z$ 
will be a non-isolated singular point of $X_1$. In the definition of nondegeneracy of linear systems, we don't require the origin to be
an isolated singular point. For example, the system $(x_1 \frac{\partial}{ \partial x_1}, x_2)$ is a nondegenerate linear system of type $(1,1)$, for 
which the origin is a non-isolated singular point. 
\end{remark}

The independent vector fields $\sqrt{-1}\tilde{Y}_i = \sqrt{-1} \sum_{i=j}^m \tilde{c}_{ij} x_j \frac{\partial}{\partial x_j} $ 
with integer coefficients $\tilde{c}_{ij}$ generate an effective linear  torus action on $\mathbb{C}^m$. Thus, up to geometric equivalence, 
the classification of complex nondegenerate linear integrable systems of type $(p,q)$
is nothing but the classification of effective linear actions of  the torus 
$\mathbb{T}^p$ on $\mathbb{C}^m$, i.e. complex linear $m$-dimensional representations of $\bbT^p$. 
The classification in the real case is more complicated: two real linear systems may be non-equivalent 
but have the same  complexification.

\section{Linearization and rigidity of nondegenerate singularities}

\begin{defn} \label{defn:Nondegenerate}
A fixed  point of an integrable system  $(X_1,\hdots, X_p, F_1,\hdots, F_q)$ of type $(p,q)$
is called {\bf nondegenerate} if its linear part
is a nondegenerate linear integrable system. A singular point of an integrable system 
is called {\bf nondegenerate} if it becomes a nondegenerate fixed point after a reduction.
\end{defn}

\begin{remark}
Though the choice of first integrals is not important in Definition \ref{defn:GeometricEquivalence} 
of geometric  equivalence, the $q$-tuple $F_1,\hdots, F_q$ of first integrals in the above definition of 
nondegeneracy is assumed to be chosen so that  not only they are functionally independent, 
but their homogeneous parts are also functionally independent.  
(According to a simple analogue of Ziglin's lemma \cite{Ziglin-Branching1982}, 
in the analytic case, such a choice is always possible).
\end{remark}

\begin{theorem}[Geometric linearization]
\label{thm:linearization-nondegenerate}
An analytic (real or complex) integrable system near a nondegenerate fixed point is locally geometrically equivalent to a nondegenerate 
linear integrable system, namely its linear part. 
\end{theorem}

\begin{proof}
The proof is a consequence of  the main results of 
\cite{Zung-Poincare2002,Zung-Birkhoff2005}, which say that if a system is analytically integrable,
then in a neighborhood of any singular point it admits a local analytic effective torus action
(the torus is a real torus but it acts in the complex space), whose dimension is equal to the so called toric
degree of the system at that point, and the linearization of this torus action is 
equivalent to the Poincaré-Dulac normalization of the system. (This torus action is intrinsic to the system and
is defined as a kind of double commutator, i.e. any vector field which commutes with the system also commutes with this torus action).
It remains to prove  that, in the nondegenerate case, the Poincaré-Dulac 
normalization is actually a geometric linearization of the system.

Indeed, in the nondegenerate complex analytic case, 
it follows directly from the definition of the toric degree (see \cite{Zung-Poincare2002} or \cite{Zung-Torus2006}), that the toric degree
at the isolated singular point is equal to $p$, and so there is an effective analytic torus action of dimension $p$ around the singular point 
which preserves the system.  
By a local diffeomorphism, we may assume that this torus action is linear and is generated by
$p$ vector fields $\sqrt{-1}\tilde{Y}_1,\hdots,\sqrt{-1}\tilde{Y}_p$, where each $\tilde{Y}_i$ is linear diagonal with integer coefficients:
$\tilde{Y}_i = \sum_{i=j}^m  \tilde{c}_{ij} x_j \frac{\partial}{\partial x_j},$ $c_{ij} \in \mathbb{Z}$ for all $i,j.$ (The Poincaré-Dulac
normalization amounts to the linearization of this torus action, see \cite{Zung-Poincare2002}).

Moreover, from the construction of this torus action we have that
$\tilde{Y}_i \wedge X_1 \wedge \hdots \wedge X_p = 0$ for all $i=1,\hdots, p$ (because the torus action also preserves the first integrals so its 
generators must be tangent to the complex common level sets of the first integrals). Since $ \tilde{Y}_1,\hdots, \tilde{Y}_p$ are independent, by 
dimensional consideration, the inverse is also true: $X_i \wedge \tilde{Y}_1 \hdots \wedge \tilde{Y}_p = 0$ for all $i=1,\hdots, p.$ 
Lemma \ref{lemma:division} below says that we can write 
$X_i = \sum_{j} f_{ij} \tilde{Y}_i$ in a unique way, where $f_{ij}$ are local analytic functions, which are also first integrals of the system.
The fact that the matrix $(f_{ij})$ is invertible, i.e. it has non-zero determinant at $z$, is also clear, because $(\tilde{Y}_1, \hdots, \tilde{Y}_p)$
are nothing but a linear transformation of the linear part of $(X_1,\hdots,X_p).$

What we have proved is that,  near a nondegenerate fixed point, an integrable system is geometrically equivalent to its linear part, at least
in the complex analytic case. In the real analytic case, the vector fields $(\tilde{Y}_1,\hdots,\tilde{Y}_p)$ are not real in general, but the proof
will remain the same after a complexification, because the Poincaré-Dulac normalization in the real case can be chosen to be real 
(see \cite{Zung-Poincare2002,Zung-Birkhoff2005}).  
\end{proof}

\begin{lemma}[Division lemma] \label{lemma:division}
If  $(Y_1,\hdots,Y_p, G_1,\hdots, G_q)$ is a nondegenerate linear integrable system, and $X$ is a local analytic vector field 
which commutes with $Y_1,\hdots,Y_p$ and such that $X \wedge Y_1 \wedge \hdots \wedge Y_p = 0$, then we can write 
$X = \sum f_i Y_i$ in a unique way, where $f_i$ are local analytic functions which are common first integrals of $Y_1,\hdots, Y_p$.
\end{lemma}

{\bf Proof}. Without loss of generality, we may assume that $Y_i = \sum_{j} c_{ij} Z_j$, where $c_{ij}$ are integers and 
$Z_i = x_i \frac{\partial}{ \partial x_i}$ in some coordinate system $(x_1,\hdots, x_m)$. 
We will write $X = \sum_i g_i Z_i,$ where $x_i g_i$ are analytic functions.
The main point is to prove that $g_i$ are analytic functions, and the rest of the lemma will follow easily. Let $\prod_i x_i^{\alpha_i}$
be a polynomial first integral of the linear system. Then we also have $X(\prod_i x_i^{\alpha_i}) = 0,$ which implies that
$\sum_i \alpha_i g_i = 0.$ If $\alpha_1 \neq 1$ then  $x_1g_1 = (-\sum_{i=2}^m x_1g_i)/\alpha_1$ vanishes when
$x_1 = 0,$ and so $x_1g_1$ is divisible by $x_1$, which means that $g_1$ is analytic. Thus, for each $i$, if we can choose
a monomial first integral $\prod_i x_i^{\alpha_i}$ such that $\alpha_i \neq 0,$ then $g_i$ is analytic. Assume now that
all monomial first integrals $\prod_i x_i^{\alpha_i}$ must have $\alpha_1 = 0.$ It means that all the first integrals are also
invariant with respect to the vector field $Z_1 = x_1 \frac{\partial}{ \partial x_1}$. Then $Z_1$ must be a linear combination of
$Y_1,\hdots, Y_p$ (because the system is already ``complete'' and one cannot add another independent commuting vector field to it),
and we have $[Z_1, X] = 0.$ From this relation it follows easily that $g_1$ is also analytic in this case. Thus, all functions $g_i$
are analytic. $\square$

Theorem \ref{thm:linearization-nondegenerate} can be extended to the case of non-fixed 
nondegenerate singular points in an obvious way, with the same proof, using our results \cite{Zung-Poincare2002,Zung-Birkhoff2005} 
on the toric characterization of local normalizations of vector fields:
 
\begin{thm} \label{thm:linearization-nondegenerate2}
 Any analytic integrable dynamical system near a nondegenerate singular point is locally geometrically 
equivalent to a direct product of a linear nondegenerate integrable system and a constant (regular) integrable system.
\end{thm}

We also have an extension of Ito's theorem \cite{Ito-Nonresonant} to the non-Hamiltonian case. Ito's theorem says that, an analytic integrable 
Hamiltonian system at a non-resonant singular point (without the requirement of nondegeneracy of the momentum map at that point) can
also be locally geometrically linearized (i.e. locally one can choose the momentum map so that the system becomes nondegenerate and geometrically
linearizable). For Hamiltonian vector fields, there are many auto-resonances due to their Hamiltonian nature, which are not counted as
resonance in the Hamiltonian case. So, in the non-Hamiltonian case, we have to replace the adjective ``non-resonant'' by ``minimally-resonant'':

\begin{defn}
A vector field $X$ in a integrable dynamical system $(X_1 = X, \hdots, X_p, F_1,\hdots, F_q)$ of type $(p,q)$ 
is called {\bf minimally resonant} at a singular point $z$ if its toric degree at $z$ is equal to $p$ (maximal possible). 
 \end{defn}

\begin{theorem}  
Minimally-resonant singular points of analytic integrable systems are also 
locally geometrically linearizable in the sense that one can change the 
auxiliary commuting vector fields (keeping the first vector field and the functions intact) 
in order to obtain a new integrable system which is locally geometrically linearizable.
\end{theorem}

\begin{proof}
The proof is similar to the proof of Theorem \ref{thm:linearization-nondegenerate} 
and is also a direct consequence of the main results of \cite{Zung-Poincare2002}. 
\end{proof}

In order to give another justification for our notion of nondegeneracy of singular points of integrable non-Hamiltonian systems,  we will also
show that such singularities are deformation rigid:

\begin{theorem}[Rigidity of nondegenerate singularities] \label{thm:rigid}
 Let 
\begin{equation}
(X_{1, \theta},\hdots, X_{p, \theta}, F_{1, \theta}, \hdots, F_{q, \theta}) 
\end{equation}
be an analytic family of integrable systems of type $(p,q)$
depending on a parameter $\theta$ which can be multi-dimensional: $\theta = (\theta_1,\hdots,\theta_s)$, 
and assume that $z_0$ is a nondegenerate fixed point when $\theta = 0$. Then there exists a local 
analytic family of fixed points $z_{\theta}$, such that $z_{\theta}$ is a fixed point of 
$(X_{1, \theta},\hdots, X_{p, \theta}, F_{1, \theta}, \hdots, F_{q, \theta})$ for each $\theta$, and moreover,
up to geometric equivalence, the local structure of  $(X_{1, \theta},\hdots, X_{p, \theta}, F_{1, \theta}, \hdots, F_{q, \theta})$ at $z_\theta$ 
does not depend on $\theta$.
\end{theorem}

\begin{proof}
We can put the integrable systems in this family together to get one ``big'' integrable system of type $(p, q+s)$, with the last
coordinates $x_{m+1}, \hdots, x_{m+s}$ as additional first integrals. Then $z_0$ is still a nondegenerate fixed point for this big integrable system,
and we can apply Theorem (\ref{thm:linearization-nondegenerate}) to get the desired result.   
\end{proof}

\section{Linearization of smooth integrable systems}

In the smooth case, we still have the same definitions of linear part, geometric equivalence, 
nondegeneracy and geometric linearization as in the analytic case. 
We have the following conjecture, which is the smooth version of Theorem \ref{thm:linearization-nondegenerate2}:

\begin{conjecture}
 Any smooth integrable dynamical system near a nondegenerate singular point is locally geometrically 
smoothly equivalent to a direct product of a linear nondegenerate integrable system and a constant system..
\end{conjecture}

We believe that the above conjecture is true, but don't have a full proof of it in the general case. We will prove it
for the case of systems of type $(n,0)$ in the next section. As a rule, normal forms results for smooth systems require more
elaborate work than for analytical systems, because of the lack of complex analytic tools. We have already seen this for
Hamiltonian systems, where the proof of Eliasson's local linearization theorem \cite{Eliasson-Normal1990}, which is the 
smooth counterpart of Vey's theorem \cite{Vey} (see also \cite{Zung-Birkhoff2005}) is much longer than the proof of 
Vey's theorem.

Let us indicate here why we believe that the above conjecture is true, and some methods which could be used to prove it.

1) By geometric arguments similar to the ones used in \cite{Zung-Symplectic1996,Zung-Poincare2002,Zung-Birkhoff2005}, we can show
the existence of a smooth torus $\bbT^d$-action which preserves the system, where $d$ is the {\bf real toric degree}
of the system (i.e. part of the toric degree whose corresponding action is real). Up to geometric equivalence, 
we can also assume that the vector fields which generate this torus action are part of our system.
The remaining vector fields of the system are hyperbolic and invariant with respect to this smooth torus action.

2) Theorem \ref{thm:linearization-nondegenerate} is also true in the formal case with the same proof, because the results
of \cite{Zung-Poincare2002,Zung-Birkhoff2005} are also true in the formal category. So we can apply a formal linearization
to our smooth system. Together with Borel's theorem, it means that there is a local smooth coordinate system in which our
system is already geometrically linear up to a flat term.

3) After the above Step 2, one can try to use  results and techniques on finite determinacy of mappings 
à la Mather \cite{Mather-Determinacy1969} to find  a matrix whose entries are smooth first integrals, such that 
when multiplying this matrix with our vector fields, we obtain a new geometrically equivalent system whose vectors
are linear + flat terms.

4) One can now try to invoke an equivariant version of Sternberg--Chen theorem \cite{Chen-Vector1963,Sternberg}, 
due to Belitskii and Kopanskii \cite{BK-Equivariant2002}, 
which says that smooth equivariant hyperbolic vector fields which are formally linearizable are also smoothly equivariantly linearizable.
Of course, we will have to do it simultaneously for all commuting hyperbolic vector fields. So we need an extension of 
the result of Belitskii and Kopanskii to the situation of a smooth 
$\bbR^k$-action with some hyperbolicity property which is formally linear. Maybe we would also need a version of 
Belitskii--Kopanskii--Sternberg--Chen for  vector fields which have first integrals. Techniques of 
\cite{Chaperon-geometrie1986,ColinVey-Morse1979, DufourMolino-AA, Eliasson-Normal1990} may also be useful here.

\section{Smooth systems of type $(n,0)$}

In this section, we consider a  smooth integrable system of type $(n,0)$, 
consisting of $n$ commuting vector fields $X_1,\hdots, X_n$
on a $n$-dimensional manifold $M^n$. (There is no function, just vector fields). In this case, a geometric linearization means
a true linearization of the vector fields, because there is no function. We will denote by 
\begin{equation}
 \rho: \bbR^n \times M^n \to M^n
\end{equation}
the (local) action of $\bbR^n$ on $M^n$ generated by these vector fields. Moreover, for each vector $v = (v^i) \in \bbR^n$,
we will denote by
\begin{equation}
 X_v = \sum_{i=1}^n v^i X_i
\end{equation}
and call it the {\bf generator of the action associated to $v$}.

First of all, we have the following classification of nondegenerate real linear systems of type $(n,0),$ or in other words, nondegenerate
linear actions of $\bbR^n$ on $\bbR^n$. Such actions are generated by Cartan subalgebras of the Lie algebra of linear vector fields
on $\bbR^n$. This Lie algebra is naturally isomorphic to  $gl(n,\bbR)$, and so the classification of nondegenerate linear actions of 
$\bbR^n$ on $\bbR^n$ corresponds to a classical classification up to conjugation of Cartan subalgebras of  $gl(n,\bbR)$:

\begin{thm} \label{thm:LinearNF}
Let $\rho^{(1)}: \bbR^n \times \bbR^n \to \bbR^n$ be a nondegenerate linear action of 
$\bbR^n$ on  $\bbR^n$.
Then there exist nonnegative integers  $h, e \geq 0$ such that $2h + e = n$, a linear coordinate system $x_1, \hdots, x_n$
on $\bbR^n$, and a linear basis $(v_1,\hdots,v_n)$ of $\bbR^n$ such that the generators $Y_i = X_{v_i} = \sum_j v_i^j X_j$
of the action $\rho^{(1)}$ with respect to the basis $(v_1,\hdots,v_n)$ can be written as follows:
\begin{equation}
\begin{cases}
Y_i = x_i\frac{\partial }{\partial x_i} \quad \forall \quad i = 1,..., h \\
Y_{h+2j-1} = x_{h+2j-1}\frac{\partial }{\partial x_{h+2j-1}} +  x_{h+2j}\frac{\partial }{\partial x_{h+2j}}  \\
Y_{h+2j} = x_{h+2j-1}\frac{\partial }{\partial x_{h+2j}} -  x_{h+2j}\frac{\partial }{\partial x_{h+2j-1}} \quad \forall \quad j = 1,..., e. 
\end{cases} 
\end{equation}
\end{thm} 

The proof of the above theorem is a simple exercise of linear 
algebra: since the linear vector fields $X_i$ commute, they are simultaneously
diagonalizable over $\bbC$. Their joint 1-dimensional real eigenspaces correspond to {\bf hyperbolic} components $Y_i$, 
while joint complex eigenspaces correspond to  components $(Y_{h+2j-1},Y_{h+2j})$, which are called {\bf elbolic}
components. (Elbolic means elliptic+hyperbolic; an elbolic component has two sub-components, one of which is elliptic and
the other one is hyperbolic).

Let $p$ be a singular point of a smooth integrable system $(X_1,\hdots, X_n)$
of type $(n,0)$, i.e. $\dim Span_\bbR (X_1(p),\hdots, X_n(p)) < n.$
We do not require $p$ to be a fixed point, i.e.
$\rank p := \dim Span_\bbR (X_1(p),\hdots, X_n(p))$ may be 0 or positive.
Recall that, if $\rank p = k > 0$, then without loss of generality, we may assume that 
$X_{n-k+1}(p)\wedge \hdots \wedge X_n(p) \neq 0,$ and $p$ will be called a
{\it nondegenerate singular point} if it becomes a nondegenerate fixed point
of a system of type $(n-k,0)$ which is obtained from the original system of type
$(n,0)$ by a local reduction with respect to the free local $\bbR^k$-action generated
by $X_{n-k+1},\hdots, X_n$ (see Definition \ref{defn:Nondegenerate}).
The main result of this section is the following local normal form theorem,
which is the smooth version of Theorem \ref{thm:linearization-nondegenerate2} for systems
of type $(n,0)$:

\begin{thm} \label{thm:NormalForm}
Let $p$ be a nondegenerate singular point of a smooth integrable system $(X_1,\hdots, X_n)$
of type $(n,0)$. Denote by
\begin{equation}
 m = n - \dim Span_\bbR (X_1(p),\hdots, X_n(p))
\end{equation}
the corank of the system at  $p$. Then there exists a smooth local coordinate 
system $(x_1,x_2,..., x_n)$ in a neighborhood of $p$, non-negative integers $h, e \geq 0$
such that $h + 2e =m$, and a basis $(v_1,\hdots, v_n)$ of $\bbR^n$ such that the corresponding
generators $Y_i = X_{v_i} (i = 1, \hdots, n)$ of $\rho$ have the following form:
\begin{equation} \label{eqn:NormalForm}
\begin{cases}
Y_i = x_i\frac{\partial }{\partial x_i} \quad \forall \quad i = 1,\hdots, h \\
Y_{h+2j-1} = x_{h+2j-1}\frac{\partial }{\partial x_{h+2j-1}} +  x_{h+2j}\frac{\partial }{\partial x_{h+2j}}  \\
Y_{h+2j} = x_{h+2j-1}\frac{\partial }{\partial x_{h+2j}} -  
      x_{h+2j}\frac{\partial }{\partial x_{h+2j-1}} \quad \forall \quad j = 1, \hdots, e \\
Y_k = \frac{\partial }{\partial x_k} \quad \forall \quad k = m+1, \hdots, n.
\end{cases}
\end{equation}
The numbers $(h, e)$ do not depend on the choice of local coordinates.
\end{thm}

\begin{proof} (See Remark \ref{remark:referee} for a different, simple proof proposed by the
referee of this paper). The fact the the numbers $(h,e)$ in the above theorem do not 
depend on the choice of coordinates is clear, because they are invariant of the Cartan subalgebra 
of the corresponding reduced system at $p$. We will call $h$ the number of hyperbolic components, 
and $e$ the number of elbolic components of the
system at $p$. We will prove the above theorem by induction on the couple $(h,e)$, 
and will divide the proof into several steps. \\

\underline{Step 1: The case when $(h,e) = (1,0)$.}

In this step, we assume that the corank of the system at $p$ is 1. 
Without loss of generality, we may assume that $X_2(p) \wedge \hdots \wedge X_n(p) \neq 0.$ 
Since the vector fields $X_1, \hdots, X_n$ commute, applying the classical Frobenius theorem, 
we can find a local coordinate system
$(y_1,\hdots, y_n)$ in which $X_i = \frac{\partial}{\partial y_i}$ for $i = 2,\hdots, n.$ In this coordinate
system, the first vector field $X_1$ will have the form:
\begin{equation}
 X_1 = f_1(y_1) \frac{\partial}{\partial y_1} + \hdots + f_n(y_1) \frac{\partial}{\partial y_n}
\end{equation}
(where the functions $f_1,\hdots, f_n$ depend only on the coordinate $y_1$, due to the fact that $X_1$
commutes with the other vector fields). Moreover, we have $f_1(0) = 0$ and $f'(0) \neq 0$, because $p$
is a nondegenerate singular point, so we can write $f_1(y_1) = g(y_1). y_1$, with $g(0) \neq 0.$ Write
$f_i (y_1) = f_i(0) + g_i(y_1). y_1$ for $i=2,\hdots, n$ also.

Replacing $X_1$ by another generator $Z_1 = X_1 - \sum_{i=2}^n f_i(0) X_i$ of the system, we can write
\begin{equation}
 Z_1 = y_1 \sum_{i=1}^n g_i(y_1) \frac{\partial}{\partial y_1} = y_1 \hat{Z_1},
\end{equation}
with $g_1(0) \neq 0$. Notice that $Z_1$ is a regular vector field. The regular integral curve $\Gamma$ of $\hat{Z}_1$
through $p$ is also an integral curve for $Z_1$, and on $\Gamma$ the vector field $Z_1$ can be linearized, i.e.
there is a coordinate function $x_1$ on $\Gamma$, such that the restriction of $Z_1$ to $\Gamma$ has the form
$Z_1 = a x_1 \frac{\partial}{\partial x_1}$, where $a$ is a non-zero constant.

Define new coordinates $(x_1,\hdots,x_n)$ by the following formulas: 
For each  point $q$ in a small neighborhood of $q$,  $x_2(q),\hdots x_n(q)$ are the unique numbers such that 
$q'  = \phi_{X_2}^{-x_2(q)} \circ \hdots \circ \phi_{X_n}^{-x_n(q)} (q)$ belongs to $\Gamma$, 
where $\phi_X$ denotes the flow of the vector field $X$, and  put $x_1(q) = x_1 (q')$. 
One then verifies easily that 
$(x_1,\hdots,x_n)$, together with $Y_1 = Z_1/a$ and $Y_i = X_i$ for all $i \geq 2$ satisfy Equations \eqref{eqn:NormalForm}. \\

\underline{Step 2: The case when $e=0$ and $h  > 1$ arbitrary.}

We will prove by induction on $h$, so let's assume that the theorem is already proved when there are $h-1$ hyperbolic components
and zero elbolic component. Consider now the case with $h$ hyperbolic components and zero elbolic component.

Invoking the formal version of Theorem \ref{thm:linearization-nondegenerate2}, we can assume, without loss of generality,
that the system is already linearized up to flat terms. In other words, we can assume that:
\begin{equation*} 
\begin{cases}
Y_i = x_i\frac{\partial }{\partial x_i} + flat \quad \forall \quad i = 1, \hdots, h \\
Y_k = \frac{\partial }{\partial x_k} + flat \quad \forall \quad k = h+1, \hdots, n,
\end{cases}
\end{equation*}
where $flat$ means a term which is flat at $p$. Since the vector fields $Y_k$ ($k \geq h+1$)
are regular and commute with each other, by the classical
Frobenius theorem we can rectify our coordinate system a bit more to 
kill the flat terms in the expression of $Y_k, k \geq h+1$, and get:
\begin{equation*} 
\begin{cases}
Y_i = x_i\frac{\partial }{\partial x_i} + flat \quad \forall \quad i = 1, \hdots, h \\
Y_k = \frac{\partial }{\partial x_k}  \quad \forall \quad k = h+1, \hdots, n.
\end{cases}
\end{equation*}

Consider the vector field
\begin{equation}
Z_1 = Y_1 - \sum_{i=2}^h Y_i. 
\end{equation}
This vector field is not hyperbolic at $p$ if $h < n$ (it has $n-h$ eigenvalues equal to 0), but it is hyperbolic for the
reduced $h$-dimensional system (the local reduction is done by forgetting about the coordinates $x_{h+1},\hdots, x_n$,
or in other words, by taking the quotient of the neighborhood of $p$ 
by the flows of the vector fields $Y_{h+1},\hdots, Y_n$). So, according to the classical stable manifold theorem, we
have a smooth $(h-1)$-dimensional stable manifold with respect to $Z_1$ on the reduced $h$-dimensional manifold,
which, when pulled back to a neighborhood of $p$ in $M^n$, becomes a smooth center-stable $(n-1)$-dimensional
manifold of $Z_1$, which we will denote by $\Sigma_1$.

Note that $\Sigma_1$ is invariant with respect to our system, which means that all the vector fields $Y_1,\hdots, Y_n$
are tangent to $\Sigma_1$, which in turn implies that the points of $\Sigma_1$ are singular with respect to our system
(the rank of the system at each point is at most $n-1$). But if we forget about $Y_1$, then $(Y_2,\hdots,Y_n)$ form
an integrable system on $\Sigma_1$ of type $(n-1,0)$ which admits $p$ as a singular point with $(h-1)$ 
hyperbolic components, so this sub-system can be linearized on $\Sigma_1$ according to our induction hypothesis. 
For the moment, we don't need this linearization, just a consequence of it which says that for any point $q \in \Sigma_1$, the
closure of the orbit through $q$ of the sub-system (i.e. of the infinitesimal $\bbR^{n-1}$-action generated by 
$(Y_2,\hdots,Y_n)$) contains $p$. With this, we can show that
\begin{equation}
Y_1 (q) = 0 \ \forall \ q \in \Sigma_1. 
\end{equation}
Indeed, if $Y_1(q) \neq 0$ then we can write $Y_1(q) = \sum_{i\geq 2} a_i Y_i(q)$, where $a_i$ are numbers and at least
one of them is different from 0. By commutativity, for any other point $q'$ on the orbit of the system through $q$, 
we also have  $Y_1(q') = \sum_{i\geq 2} a_i Y_i(q') = \sum_{i =2}^h a_i x_i\frac{\partial }{\partial x_i} + 
\sum_{k=h+1}^n a_k \frac{\partial }{\partial x_k} + \hdots$. But when $q'$ is very close to $q$, this expression
contradicts the expression $Y_1 =  x_1\frac{\partial }{\partial x_i} + flat.$ So we must have  $Y_1(q) = 0.$

It is now easy to see that we can write
\begin{equation}
\Sigma_1 = \{q \in \cU \ | \ Y_1 (q) = 0 \}, 
\end{equation}
where $\cU$ denotes a small neighborhood of $p$. Moreover, by construction, $\Sigma_1$ is tangent to $\{x_1=0\}$ at $p$.
By a smooth change of coordinates, we can assume that $\Sigma_1 = \{x_1 = 0\}$. Do the same thing for every $i=1,\hdots, h$.
We can now assume that for every $i=1,\hdots, h$ we have
\begin{equation}
 \Sigma_i = \{ q \in \cU \ | \ Y_i(q) = 0\} = \{x_i = 0\}.
\end{equation}
Then we can write 
\begin{equation}
 Y_i = x_i \hat{Y}_i,
\end{equation}
where $\hat{Y}_i$ is a regular vector field for each $i=1,\hdots,h$. 

Construct a new coordinate system $(y_1,\hdots,y_n)$ as follows.

For each $i = 1,\hdots, h$:

On the regular integral curve $\Gamma_i$ of the vector field $\hat{Y}_i$ through $p$, let $y_i$ be a coordinate function which
linearizes $Y_i$: the restriction of $Y_i$ to $\Gamma_i$ has the form $Y_i = y_i\frac{\partial }{\partial y_i}$.
The vector fields $\hat{Y}_j, j \neq i$ and $Y_{h+1},\hdots, Y_n$ 
satisfy the integrability condition of Frobenius and generate a regular foliation of codimension 1, 
which we will denote by $\cF_1$. For each point $q$ in a small neighborhood $\cU$ of $p$,
define $y_i(q)  = y_i(q')$, where $q'$ is the intersection of the leaf of  $\cF_i$ through $q$ with $\Gamma$.

For  the other indices:

The vector fields $\hat{Y}_1,\hdots, \hat{Y}_h$ generate a regular $h$-dimensional foliation. Denote by
$\Gamma$ the leaf of that foliation through $p$. The functions $y_{h+1}(q), \hdots, y_n(q)$ are defined by the 
condition:
$$
\phi_{X_{h+1}}^{-y_{h+1}(q)} \circ \hdots \circ \phi_{X_n}^{-y_n(q)} (q) \in \Gamma.
$$
One then verifies easily that the vector fields $Y_1,\hdots, Y_n$ satisfy Equations \eqref{eqn:NormalForm}
with respect to the new coordinate system $(y_1,\hdots,y_n).$ \\

\underline{Step 3: The case when $(h,e) = (0,1)$.}

In this case, using formal linearization, we obtain a 
local smooth coordinate system $(x_1,\hdots, x_n)$ in which we have:
\begin{equation*}
\begin{cases}
Y_{1} = x_{1}\frac{\partial }{\partial x_{1}} +  x_{2}\frac{\partial }{\partial x_{2}} + flat \\
Y_{2} = x_{1}\frac{\partial }{\partial x_{2}} -  x_{2}\frac{\partial }{\partial x_{1}}  + flat \\
Y_k = \frac{\partial }{\partial x_k} + flat \quad \forall \  k = 3, \hdots, n.
\end{cases}
\end{equation*}
Using geometric arguments similar to the ones in 
\cite{Zung-Symplectic1996,Zung-Poincare2002,Zung-Birkhoff2005} for constructing torus actions, we can
assume that $Y_2$ generates an action of $\bbT^1.$ Invoking Bochner's linearization theorem, we can assume
that $Y_2$ is already linear, i.e. the flat term in its expression is actually 0:
\begin{equation*}
\begin{cases}
Y_{1} = x_{1}\frac{\partial }{\partial x_{1}} +  x_{2}\frac{\partial }{\partial x_{2}} + flat \\
Y_{2} = x_{1}\frac{\partial }{\partial x_{2}} -  x_{2}\frac{\partial }{\partial x_{1}}  \\
Y_k = \frac{\partial }{\partial x_k} + flat \quad \forall \  k = 3, \hdots, n.
\end{cases}
\end{equation*}

Using arguments similar to those in Step 2, one can show that the center manifold $\Sigma$ of $Y_1$
is a smooth submanifold of dimension $n-2$, and $Y_1$ vanishes on it, i.e. we can write
$\Sigma = \{q \in \cU \ | \ Y_1(q) = 0\},$  where $\cU$ denotes a small neighborhood of $p.$
$\Sigma$ is also the set of fixed points of the $\bbT^1$-action generated by $Y_2,$ and so we have
$$\Sigma = \{q \in \cU \ | \ Y_1(q) = 0\} = \{q \in \cU \ | \ Y_2(q) = 0\} = \{x_1=x_2 = 0\}.$$

One then prove easily that there is a unique local 2-dimensional surface $\Gamma$ which contains $q$
and which is invariant with respect to $Y_1$ and $Y_2$. On $\Gamma$, there is a coordinate system
$(y_1,y_2)$ with respect to which the restrictions of  $Y_1$ and $Y_2$ to $\Gamma$ are linear. One then
proceed as in Step 1 to construct a new coordinate system $(y_1,\hdots, y_n)$ in which the vector fields
$Y_1,\hdots, Y_n$ satisfy Equations \eqref{eqn:NormalForm}. \\

\underline{Step 4: The general case, with arbitrary $(h,e)$}

It is just a combination of the arguments presented in the previous three steps. In fact, one can treat elbolic components
in almost the same way as hyperbolic components, except that instead of integral curves one has to use integral 2-dimensional
disks, and instead of codimension-1 manifolds on which the corresponding vector fields vanish one has to use
codimension-2 submanifolds for elbolic components.
\end{proof}

\begin{remark}
\label{remark:referee}
Another,  simpler proof of Theorem \ref{thm:NormalForm} along the following lines was communicated to us by the referee: 

i) In the case when $m=n$, the linear part of an appropriate  linear combination $E = \sum a_i X_i$
of the vector fields $X_1,\hdots, X_n$ is a radial vector field, i.e. has the form 
$E^{(1)} = \sum_{i=1}^n x_i \frac{\partial}{\partial x_i}$. By Sternberg's theorem, $E$ is smoothly
linearizable, i.e. we can assume that $E = \sum_{i=1}^n x_i \frac{\partial}{\partial x_i}$ after a smooth
change of the coordinate system. Since the vector fields $X_i$ commute with the radial vector field 
$E = \sum_{i=1}^n x_i \frac{\partial}{\partial x_i}$, they are automatically linear in the new coordinate
system. 

ii) The case with $m < n$ can be reduced to the above case, 
by considering the $m$-dimensional isotropy algebra of the infinitesimal $\mathbb{R}^n$-action 
at the singular point, and showing the existence of a
$m$-dimensional invariant submanifolds of the subaction of this isotropy algebra, 
which is transverse to the local orbit through the singular point 
of the $\mathbb{R}^n$-action.

This new prof also works for finitely differentiable systems.
\end{remark}

\vspace{0.5cm}

{\bf Acknowledgements}. I would like to thank the referee for many critical remarks which helped me improve
the presentation of this paper, and especially for letting me know his proof of Theorem \ref{thm:NormalForm}.


\vspace{0.5cm}

\begin{thebibliography}{10}
\baselineskip0.4cm
\parskip-0.1cm

\bibitem{AyoulZung-Galois2010}
M. Ayoul, N.T. Zung, {\it Galoisian obstructions to non-Hamiltonian integrability}, Comptes Rendus Mathématiques,
Volume 348 (2010), Issue 23,  1323-1326.

\bibitem{BaCu-Nonholonomic1999}
L. Bates and R. Cushman, \emph{{What is a completely integrable
  nonholonomic dynamical system ?}}, Reports Math. Phys. \textbf{44} (1999),
  no.~1-2, 29--35.

\bibitem{BK-Equivariant2002}
G.R. Belitskii, A.Y. Kopanskii, {\it Equivariant Sternberg-Chen theorem}, 
Journal of Dynamics and Differential Equations, Volume 14 (2002), Number 2, pp. 349--367.

\bibitem{Bogoyavlenskij-Integrability1998}
O.I. Bogoyavlenskij, \emph{Extended integrability and bi-hamiltonian
systems},  Comm. Math. Phys. \textbf{196} (1998), no.~1, 19--51.

\bibitem{Chaperon-geometrie1986}
M. Chaperon, Géométrie différentielle et singularités de systèmes dynamiques. 
Astérisque No. 138-139 (1986), 440 pp.

\bibitem{Chen-Vector1963}
K.T. Chen, {\it Equivalence and decomposition of vector fields about an elementary critical
point}, Amer. J. Math., 85 (1963), 693--722.

\bibitem{ColinVey-Morse1979}
Y. Colin de Verdier, J. Vey, \emph{Le lemme de Morse isochore}, Topology, 18 (1979), 283-293.

\bibitem{DufourMolino-AA}
J.P. Dufour, P. Molino, {\it Compactification d'actions de $\bbR^n$ et variables action-angle avec singularités},
Publications du Département de Mathématiques (Lyon), 1988, No. 1B, 161--183.


\bibitem{Eliasson-Normal1990}
L. H. Eliasson, {\it Normal forms for Hamiltonian systems with Poisson commuting integrals
elliptic case}, Comment. Math. Helv. 65 (1990), no. 1, 4–35.

\bibitem{FedorovKozlov-Suslin}
Yu. N. Fedorov, V. V. Kozlov, {\it Various aspects of n-dimensional rigid body dynamics}, 
Trans. Am. Math. Soc. Ser. 2. V. 168. 1995. P. 141–171.

\bibitem{Ito-Nonresonant}
H. Ito, {\it Convergence of Birkhoff normal forms for integrable systems}, 
Comment. Math. Helv. 64 (1989), no. 3, 412--461.

\bibitem{Mather-Determinacy1969}
J. Mather, {\it Stability of  $C^\infty$ mappings, III. Finitely determined map-germs}, Publ.
Math. I. H. E. S. 35 (1969), 127–156.

\bibitem{Sternberg}
S. Sternberg, {\it On the structure of local homeomorphisms of Euclidean 
$n$-space, II}, Amer. J. of Math., 80 (1958), 623-631.

\bibitem{Stolovitch1} L. Stolovitch, {\it Singular complete integrability},
Publications IHES, 91 (2000), 134-210.

\bibitem{Vey}
J. Vey, \emph{Sur certains systèmes dynamiques séparables}, Amer. J. Math. 100 (1978), no. 3, 591--614.

\bibitem{Ziglin-Branching1982}
S.L. Ziglin, {\it Branching of solutions and non-existence of first integrals in Hamiltonian mechanics},
Funcional Anal. Appl. 16 (1982), 181--189. 

\bibitem{Zung-Symplectic1996}
N. T. Zung, {\it Symplectic topology of integrable Hamiltonian systems. 
I. Arnold-Liouville with singularities}, Compositio Math. 101 (1996), no. 2, 179-215.

\bibitem{Zung-Poincare2002}
N.T. Zung, \emph{Convergence versus integrability in Poincaré-Dulac normal forms}, 
Math. Res. Lett. 9 (2002), 217-228.

\bibitem{Zung-Birkhoff2005}
N.T. Zung, \emph{Convergence versus integrability in Birkhoff normal forms}, 
Ann. of Math. (2) 161 (2005), no. 1, 141–156.

\bibitem{Zung-Torus2006}
N.T. Zung, \emph{Torus actions and integrable systems},  in Topological Methods in the Theory of Integrable Systems,
Editors A.V. Bolsinov, A.T. Fomenko and A.A. Oshemkov, Cambridge Scientific Publications, 2006, 289--328.

\bibitem{ZungMinh_Rn2012}
N.T. Zung, N.V. Minh, \emph{Geometry of nondegenerate $\bbR^n$-actions on $n$-manifolds,} preprint arxiv:1203.2765 (2012),
to appear in J. Math. Soc. Japan.
\end{thebibliography}
\end{document}